\documentclass[11pt]{article}

\usepackage{amsthm,amsmath,amssymb,amsfonts,mathtools}
\usepackage{array,longtable,rotating}
\usepackage{graphicx}
\usepackage{xcolor}
\usepackage{tikz}
\usepackage{tikz-cd}
\usepackage{ytableau}
\usepackage{cancel}
\usepackage{listings}
\usepackage{pythonhighlight}
\usepackage{cite}
\usepackage[colorlinks,allcolors=blue]{hyperref}

\usepackage{ytableau}
\usepackage{mathtools}
\usepackage{listings}
\usepackage{longtable}
\usepackage{pythonhighlight}
\usepackage{cite}
\usepackage{rotating}

\newtheorem{theo}{Theorem}[section]
\newtheorem{prop}[theo]{Proposition}

\newtheorem{lemm}[theo]{Lemma}
\newtheorem{coro}[theo]{Corollary}
\newtheorem{rema}[theo]{Remark}
\newtheorem{defi}{Definition}[section]
\newtheorem{ex}[theo]{Example}

\title{\bf Steenrod Lengths and a Problem of Vakil}
\author{Duc-Khanh Nguyen}

\date{}

\newcommand{\brown}[1]{\textcolor{brown}{#1}}

\newfont{\gothic}{eufb10}

\begin{document}
\maketitle
 
\begin{abstract}
We give an explicit combinatorial description of the function $f(n)$ governing the Steenrod length of real projective spaces $\mathbb{RP}^n$. This function arises in stable homotopy theory through the action of Steenrod squares on mod-$2$ cohomology and is closely related to the ghost length, which measures the minimal number of spheres required to construct a space up to homotopy. Building on the directed graphs $T_n$ introduced by Vakil to encode degree constraints for Steenrod operations, we interpret $f(n)$ as the length of the longest directed path starting at $n$. Using this framework, we resolve a question posed by Vakil by deriving concrete combinatorial formulas for $f(n)$ in terms of binary classes and a distinguished family of integers, which we call Vakil numbers.
\end{abstract}

\textit{\\2020 Mathematics Subject Classification:} Primary 55P42; Secondary 05C20, 55S10, 68W30.\\
\textit{Keywords and phrases:} Steenrod length, ghost length, real projective spaces, directed graphs, binary expansions, combinatorial invariants.

\setcounter{tocdepth}{1}

\section{Introduction} 
We consider the directed graph whose vertices are the non-negative integers. From each vertex $n$, there is an outgoing edge to $n - 2^s$ whenever the $2^s$-bit in the binary expansion of $n$ is unset (i.e., the $s$-th bit is $0$). Let $T_n$ denote the connected graph consisting of all vertices reachable from $n$ (including $n$ itself). We define $f(n)$ to be the length of the longest directed path in $T_n$ that starts at $n$. The main purpose of this paper is to determine an explicit combinatorial formula for the function $f(n)$.

\begin{ex}\label{1} The graph $T_{10}$ has the following structure:
$$
\begin{tikzcd}
{10=1010_{(2)}} \arrow{r}{} \arrow[swap]{d}{} & 9 = 1001_{(2)} \arrow{r}{} \arrow[swap]{d}{} & 7 = 111_{(2)} \\
6 = 110_{(2)} \arrow{r}{}& 5 = 101_{(2)} \arrow{r}{} & 3 = 11_{(2)}
\end{tikzcd}
$$
    
We have $f(10)=3$. The longest paths are $10 \rightarrow 6 \rightarrow 5 \rightarrow 3$ and  $10 \rightarrow 9 \rightarrow 5 \rightarrow 3$.
\end{ex}

The motivation for this problem originates in homotopy theory. For an object $X$ (such as a CW-complex) in the stable homotopy category, the Steenrod length (mod 2) of $X$ is defined to be at least the length of the longest chain of non-trivial Steenrod operations acting on its mod-2 cohomology. A closely related invariant is the ghost length of $X$, which counts the number of wedges of sphere needed to construct $X$ up to homotopy. Both invariants serve as measures of the complexity of $X$. Christensen~\cite{christensen1998ideals} proved that the ghost length is always bounded below by the Steenrod length. 

The combinatorial function $f(n)$ admits a natural interpretation in this context: at the prime $2$, Steenrod operations $Sq^{2^s}$ act subject to binary and degree constraints, and the directed graph defining $T_n$ encodes precisely the degrees in which such operations may act nontrivially, with an edge $n\to n-2^s$ corresponding to a potentially nonzero action of $Sq^{2^s}$ in degree $n$. Consequently, $f(n)$ measures the maximum length of a chain of nontrivial Steenrod operations that can be composed starting in degree $n$.

In the special case of real projective spaces $\mathbb{RP}^n$, the Steenrod length equals $g(n)+1$, where $g(n) = \max_{q \leq n} f(q)$. Computations show that for $2 \leq n \leq 19$, the Steenrod length and ghost length coincide. In his influential paper~\cite{vakil1999steenrod}, Vakil established a formula for $g(n)$ (his Theorem 2) using recursive methods, but left open the question of finding an explicit combinatorial formula for the underlying function $f(n)$. The present work provides a complete answer to this long-standing open problem by establishing explicit combinatorial formulas for $f(n)$. Our main results are stated below.

For each non-negative integer $n$, we define the \brown{binary class} of $n$ to be the set of all non-negative integers $m$ that agree with $n$ in their binary representation after all trailing $1$'s have been removed. This equivalence class has a unique minimal element, which we denote by $\overline{n}$. We express $\overline{n}$ in binary as alternating runs of $1$'s and $0$'s:
$$ \overline{n}= \underbrace{1\cdots 1}_A 0 \underbrace{ 0\cdots 0}_a \quad \underbrace{1\cdots 1}_B 0 \underbrace{ 0\cdots 0}_b \quad \cdots \quad  \underbrace{1\cdots 1}_C 0 \underbrace{ 0\cdots 0}_c\,_{(2)}.$$ Equivalently, we may represent $\overline{n}$ in \brown{bracket form}:
$$\overline{n}=[A,0_1,\cdots, 0_a, B, 0_1,\cdots, b_b, \cdots, C,0_1,\cdots , 0_c],$$
where $A,B,\dots,C \geq 1$, and the indices $0_i$ label the individual zero bits within each block $\underbrace{0\dots0}_a$ from left to right. In general, $\overline{n}$ has form $[\alpha_k, \dots, \alpha_1]$ $(\alpha_i \geq 0)$. If $\alpha_k >0$, we say that $\overline{n}$ is a \brown{$k$-dimensional binary class}. Define $$S(\overline{n}) = \sum_{j=1}^k j \alpha_j \qquad \text{and} \qquad \Delta^n = f(n) - S(\overline{n}). $$
 Finally, we say that $n$ is a \brown{Vakil number} with \brown{Vakil pair} $(a,k)$ if $\overline{n} = 2^{a+1} k$ for some integers $a \geq 0$ and $k \geq 1$ satisfying $a \leq k \leq 2a + 1$. 

Our main results are as follows: First, we establish a simple formula for $f(n)$ when $n$ is a Vakil number (Lemma~\ref{for_Vakil}). Next, we provide explicit and straightforward expressions for $f(n)$ that apply to a large family of binary classes, including all low-dimensional cases and many higher-dimensional ones (Theorems~\ref{main-1} and \ref{main-2}). Finally, we give a complete general formula that determines $f(n)$ for every non-negative integer $n$ (Theorem~\ref{main-0}).

\begin{lemm}\label{for_Vakil} 
Let $n$ be a Vakil number with Vakil pair $(a,k)$. Then $f(n)=k+\frac{a(a+1)}{2}$.
\end{lemm}

\begin{theo}\label{main-1} Let $n$ be a number of $k$-dimensional binary class. Then $f(n)=S(\overline{n})+\Delta_k$, where $\Delta_k= 0,0,0,1,2,4,7$ for $k=1,2,3,4,5,6,7$, respectively.
\end{theo}

For a real number $r$, we write $z\approx r$ to mean that $z$ is the largest integer not bigger than $r$. 

\begin{theo}\label{main-2} Let $n$ be a number of $k$-dimensional binary class $[\alpha_k, \dots, \alpha_1]$ $(k\geq 4)$. Set $m=\lfloor \log_2k \rfloor -1$. For all $\alpha_k \geq m$, we have 
$$f(n) = S(\overline{n}) + (2^m-1)2^h+\frac{(k-h)(k-h-1)}{2} -mk,$$ where $h\geq 1$ such that $(2^m-1)2^{h-1}+h+1 \approx k$, or equivalently, $k \approx (2^m-1)2^h+h+1$.
\end{theo}

Define $T_{\overline{n}}$ to be the graph obtained from $T_n$ by identifying each vertex with the minimal element of its binary class. We write $\overline{n} \leadsto \overline{n'}$ if there is a directed path from $\overline{n}$ to $\overline{n'}$ in $T_{\overline{n}}$, and say that $\overline{n}$ \brown{can reach} $\overline{n'}$. If $\overline{n} \leadsto \overline{n'} \leadsto \overline{n''}$, then we say that $\overline{n}$ is \brown{closer} to $\overline{n'}$ than to $\overline{n''}$.

\begin{theo}\label{main-0} Let $n$ be a non-negative integer. We have $f(n) = S(\overline{n}) + \Delta^{\widehat{n}}$, where $\widehat{n}$ is the closest Vakil number to $\overline{n}$ with Vakil pair $(a,k)$, $4|k$.
\end{theo}

Although Theorem~\ref{main-0} provides the most general and comprehensive expression for $f(n)$, it is typically most efficient---in actual computations---to first apply Lemma~\ref{for_Vakil} and Theorems~\ref{main-1}--\ref{main-2} whenever possible. The procedure for identifying the appropriate $\widehat{n}$ is illustrated visually and step-by-step in the final section of the paper.

The paper is organized as follows. In Section~\ref{pre} we review Vakil's foundational results, formally introduce the notion of binary classes, present SageMath code used to computationally verify our main theorems, and establish several key technical lemmas and facts concerning canonical paths that are essential for computing $f(n)$. Section~\ref{proof} contains the complete proofs of the main results. Finally, Section~\ref{example} offers detailed worked examples demonstrating how to apply the main results in practice. To enhance clarity and accessibility, each subsection concludes with concrete illustrative examples that reinforce the definitions, lemmas, and propositions introduced therein.

\textbf{Acknowledgments:} 
The author is deeply grateful to Professor Minh-Ha Le for introducing him to this problem. He warmly acknowledges the Visiting Fellowship supported by MathCoRe and hosted by Professor Petra Schwer at Otto-von-Guericke-Universit\"at Magdeburg.

\section{Preliminaries}\label{pre}
\subsection{Review of Vakil's results} \label{Ravi result}
In this subsection we recall some key results from Vakil~\cite{vakil1999steenrod}.

\begin{lemm}[{\cite[Lemma 1]{vakil1999steenrod}}]\label{2}
Every non-negative integer $n$ can be uniquely written as $n = m 2^p + k$
with integers $m=:m(p)$, $p=:p(n)$, $k=:k(n)$ such that $p \geq 1$, $p-1 \leq m \leq 2p-1$, and $0 \leq k < 2^p$.
\end{lemm} 

\begin{defi}
The representation in Lemma \ref{2} is called the \brown{proper form} of $n$.  
\end{defi}

\begin{theo}[{\cite[Theorem 2, Corollary 3]{vakil1999steenrod}}]\label{3}
Set $g(n) = \max_{q\leq n} f(q)$. If $n=m2^p+k$ is the proper form of a non-negative integer $n$, then $g(n)=\frac{p(p-1)}{2}+m$.
The frequency table of $g(n)$ is a list of non-decreasing powers of $2$, where $2^a$ appears $a+1$ times ($a\geq 1$). 
\end{theo}

\begin{table}[ht!]
\centering
\begin{tabular}{|c|l l l l l l l l l l l l l l l l l l l l|}
\hline
$n$     &0 & 1& 2& 3& 4& 5& 6& 7& 8& 9& l0& 11& l2& l3& l4&15&16&17&18&19\\
\hline
$f(n)$ & 0 & 0& 1& 0& 2& 1& 2& 0& 3& 2& 3& 1& 4& 2& 3&0&5&3&4&2\\
\hline
$g(n)$ & 0 & 0& 1& 1& 2& 2& 2& 2& 3& 3& 3& 3& 4& 4& 4&4&5&5&5&5\\
\hline
\end{tabular}\caption{Table of $f(n)$ and $g(n)$ for small $n$.}
\end{table}

\begin{table}[ht!]
\centering{
\begin{tabular}{|c|l l l l l l l l l l l l l l l|}
\hline
$s$ &0 & 1& 2& 3& 4& 5& 6& 7& 8& 9& l0& 11& l2& l3& l4\\
\hline
$\#\{n\mid g(n)=s\}$ & 2 & 2& 4& 4& 4& 8& 8& 8& 8& l6& 16& l6& l6& l6& 32\\
\hline
\end{tabular}}
\caption{Frequency table of $g(n)$ for small $n$.} 
\end{table}

\begin{defi}
Given a non-negative integer $n$, we can uniquely divide its binary representation into three parts $\alpha$, $\beta$, $\gamma$ as follows: $\gamma$ is the block of $1$'s at the end of $n$, and $l(\beta)-1 \leq \alpha \leq 2l(\beta)-1$ where $l(\beta)$ is the number of digits of $\beta$, and $\alpha$ is interpreted as an integer. For $n$ not of the form $2^t-1$, we define the \brown{canonical edge} from vertex $n$ to be the edge corresponding to the leftmost zero in the rightmost block of zeros in the part $\beta$. A canonical edge has the form $n \longrightarrow n-2^s$ for a unique integer $s$. We call this number the \brown{canonical index of $n$}, and denote it by $s(n)$. We call the path in $T_n$ consisting of canonical edges starting from $n$ the \brown{canonical path}.
\end{defi}

\begin{rema}[{\cite[page 420]{vakil1999steenrod}}]\label{f_by_cano} For $n$ not of the form $2^t - 1$, $f(n)$ equals the length of the canonical path in $T_n$.
\end{rema}

\begin{ex}\label{5} Let $n=473=111011001_{(2)}$. We have $\alpha=111_{(2)}=7$, $\beta = 01100$, $\gamma=1$, $s(n)=2$. The first canonical edge is $473 \longrightarrow 469$. The canonical path in $T_{473}$ is $473\longrightarrow 469 \longrightarrow 467 \longrightarrow 659 \longrightarrow 651 \longrightarrow 619 \longrightarrow 603 \longrightarrow 595 \longrightarrow 431 \longrightarrow 399 \longrightarrow 383 \longrightarrow 319 \longrightarrow 255 \longrightarrow 127$. Hence, $f(473)=13$.
\end{ex}

\subsection{Binary classes}
In this subsection, we introduce the notion of the \emph{binary class} of a non-negative integer and reduce the original problem to the corresponding problem on its binary class.

\begin{defi}
Let $n$ be a non-negative integer. The \brown{binary class} of $n$ is the set of all non-negative integers $m$ whose binary representation agrees with that of $n$ after removing any trailing sequence of consecutive $1$'s (if present). Each binary class contains a unique minimal element, which we denote by $\overline{n}$ and call the \emph{canonical representative} of the class.
\end{defi}

\begin{lemm}\label{f(n)=f(m)}
If $n$ and $m$ are in the same binary class, then $f(n)=f(m)$.
\end{lemm}

\begin{proof}
	First, we see that each step $n \to n-2^s$ can be performed on the binary representation of $n$ as follows:
\begin{enumerate}
\item Define the region from $0$ at position $s$ to the first $1$ on the left.
\item Interchange all $1$'s to $0$'s and $0$'s to $1$'s in this region.
\end{enumerate}

Indeed, suppose that $n = 2^a + \cdots + 2^b + 2^c + \cdots + 2^d$ with $a > \cdots > b \geq s > c > \cdots > d$.  If $s < b$, then $n - 2^s = 2^a + \cdots + 2^{b-1} + \cdots + 2^s + 2^c + \cdots + 2^d$.  If $s = b$, we just need to remove $2^b$ from $n$.  

Second, since the tail $1\cdots1$ in the binary representation has no $0$, deleting the tail $1\cdots1$ from the binary form of $n$ does not change the value of $f(n)$.
\end{proof}

\begin{defi}
Let 
$$
n = \underbrace{1\cdots 1}_A 0 \underbrace{0\cdots 0}_a \quad 
\underbrace{1\cdots 1}_B 0 \underbrace{0\cdots 0}_b \quad 
\cdots \quad 
\underbrace{1\cdots 1}_C 0 \underbrace{0\cdots 0}_c \underbrace{1\cdots 1}_D\,_{(2)}.
$$
We rewrite $n$ in \brown{bracket form} as
$$
n = [A, 0_1, \dots, 0_a, B, 0_1, \dots, 0_b, \dots, C, 0_1, \dots, 0_c] D.
$$
Here $A, B, \dots, C \geq 1$, $D \geq 0$, the $0_i$ between $A$ and $B$ denote the individual zeroes in the block $\underbrace{0\dots 0}_a$ counted from left to right, and similarly for the other blocks. We call $D$ the \brown{tail} of $n$ and denote it by $t_n$. When $t_n = 0$, we omit writing the tail. 

We say that a binary class has \brown{dimension $k$} if it has the form
$$
[0, \dots, 0, \alpha_k, \alpha_{k-1}, \dots, \alpha_1] = [\alpha_k, \alpha_{k-1}, \dots, \alpha_1] \quad (\alpha_i \in \mathbb{Z}_{\geq 0},\ \alpha_k \geq 1).
$$
\end{defi}

Let $T_{\overline{n}}$ be the graph obtained from $T_n$ by replacing each vertex by the minimum number in its binary class. We call $T_{\overline{n}}$ the \brown{binary class graph} of $T_n$. By Lemma \ref{f(n)=f(m)}, we just need to study $T_{\overline{n}}$. Each vertex of $T_{\overline{n}}$ can be represented in integer, binary, or bracket form.

\begin{ex}\label{6}
In Example \ref{5}, after removing all trailing $1$'s from $111011001_{(2)}$, we obtain $11101100_{(2)} = 236$. Hence, $\overline{473} = 236$.  
In bracket form, we have $473 = [3,2,0_1]1$, $t_{473} = 1$, and $236 = [3,2,0_1]$, $t_{236} = 0$.  
The binary class has dimension $3$. We see that $f(236) = 13 = f(473)$, with the canonical path in $T_{\overline{236}}$ given as follows:
\begin{itemize}
    \item in integer form: $236 \to 234 \to 116 \to 114 \to 56 \to 52 \to 50 \to 24 \to 20 \to 18 \to 8 \to 6 \to 2 \to 0$,
    \item in binary form: $11101100_{(2)} \to 11101010_{(2)} \to 1110100_{(2)} \to 1110010_{(2)} \to 111000_{(2)} \to 110100_{(2)} \to 110010_{(2)} \to 11000_{(2)} \to 10100_{(2)} \to 10010_{(2)} \to 1000_{(2)} \to 110_{(2)} \to 10_{(2)} \to 0_{(2)}$,
    \item in bracket form: $[3,2,0] \to [3,1,1] \to [3,1,0] \to [3,0,1] \to [3,0,0] \to [2,1,0] \to [2,0,1] \to [2,0,0] \to [1,1,0] \to [1,0,1] \to [1,0,0] \to [2] \to [1] \to [0]$.
\end{itemize}

The binary class graph $T_{\overline{10}}$ of $T_{10}$ from Example \ref{1} is
$$
\begin{tikzcd}
{10=1010_{(2)}} = [1,1] \arrow{r}{} \arrow[swap]{d}{} & 4 = 100\cancel{1}_{(2)} = [1,0] \arrow{r}{} \arrow[swap]{d}{} & 0 = \cancel{111}_{(2)} = [0] \arrow[equal]{d} \\
6 = 110_{(2)} = [2] \arrow{r}{}& 2 = 10\cancel{1}_{(2)} = [1] \arrow{r}{} & 0 = \cancel{11}_{(2)} = [0]
\end{tikzcd}
$$
\end{ex}

\subsection{A program for $f(n)$}
Based on Remark \ref{f_by_cano}, we implemented a SageMath program \cite{SageMath} to compute $f(n)$. Here, \texttt{r(n)} returns the minimum number in the binary class $\overline{n}$, \texttt{b(n)} returns the part $\beta$, \texttt{s(n)} returns the canonical index $s(n)$, and \texttt{f(n)} returns the value $f(n)$. For example, \texttt{f(473)} returns $13$.

\begin{verbatim}
def r(n):
    B = '0' + bin(n)[2:]
    while B[-1] == '1':
        B = B[:-1]
    return int(B, 2)

def b(n):
    L = len(bin(n)[2:]); i = 0
    while i < L:
        if i-1 <= n // 2**i <= 2*i - 1:
            break
        else:
            i = i + 1
    return bin(n)[2:][L-i:]

def s(n):
    i = 0; L = len(b(n)); B = '0' + b(n)
    while B[-1] == '0' and i < L:
        i = i + 1; B = B[:-1]
    return i - 1

def f(n):
    l = 0; n = r(n)
    while n > 0:
        l = l + 1; n = r(n - 2**s(n))
    return l
\end{verbatim}

\subsection{Key lemmas}
We need the following lemmas to compute $f(n)$.

\begin{lemm}\label{8}
Let $[\alpha_k',\dots,\alpha_1'] \in T_{[\alpha_k,\dots,\alpha_1]}$ $(\alpha_k,\alpha_k' \geq 1)$. The length of any path from $[\alpha_k,\dots,\alpha_1]$ to $[\alpha_k',\dots,\alpha_1']$ is $\sum\limits_{j=1}^k j(\alpha_j - \alpha_j')$.
\end{lemm}

\begin{proof}
Since $\alpha_k, \alpha_k' \geq 1$, the two classes have the same dimension. Hence, each step in a path from $[\alpha_k,\dots,\alpha_1]$ to $[\alpha_k',\dots,\alpha_1']$ has the form
$$
[\beta_k,\dots,\beta_j,\beta_{j-1},\dots,\beta_1] \to [\beta_k,\dots,\beta_j-1,\beta_{j-1}+1,\dots,\beta_1].
$$
Consider $\beta_j$ as the exponent of $x_j$ in the monomial $x_k^{\beta_k} \cdots x_1^{\beta_1}$. Then each step that does not change the dimension of the binary class is equivalent to multiplication by one of $x_j / x_{j+1}$ for $j \in [1,k-1]$, or by $1/x_1$.  

If
$$
[\alpha_k,\dots,\alpha_1] \stackrel{t \text{ times}}{\longrightarrow} [\alpha_k',\dots,\alpha_1'],
$$
then $t = t_1 + \cdots + t_k$, where $t_j \geq 0$ and
$$
\prod_{j=1}^k x_j^{\alpha_j} \cdot \left( \frac{x_{k-1}}{x_k} \right)^{t_k} \cdots \left( \frac{x_1}{x_2} \right)^{t_2} \left( \frac{1}{x_1} \right)^{t_1} = \prod_{j=1}^k x_j^{\alpha_j'}.
$$
Or equivalently,
\begin{align*}
\alpha_k - t_k &= \alpha_k', \\
\alpha_{k-1} + t_k - t_{k-1} &= \alpha_{k-1}', \\
\alpha_{k-2} + t_{k-1} - t_{k-2} &= \alpha_{k-2}', \\
&\vdots \\
\alpha_2 + t_3 - t_2 &= \alpha_2', \\
\alpha_1 + t_2 - t_1 &= \alpha_1'.
\end{align*}
We have
$$
t = \sum_{j=1}^k t_j = k t_k + \sum_{j=1}^{k-1} j (t_j - t_{j+1}) = \sum_{j=1}^k j (\alpha_j - \alpha_j').
$$
Since $t$ is the length of the path $[\alpha_k,\dots,\alpha_1] \to \cdots \to [\alpha_k',\dots,\alpha_1']$, the conclusion follows.
\end{proof}

Suppose that $\overline{N} = [\alpha_k, \dots, \alpha_1]$. Set $S(N) = \sum\limits_{j=1}^k j \alpha_j$ and $\Delta^N = f(N) - S(N)$. If there is a path from $\overline{N}$ to $\overline{N'}$ in $T_{\overline{N}}$, we say that $\overline{N}$ \brown{can reach} $\overline{N'}$ and write $\overline{N} \leadsto \overline{N'}$. If $\overline{N} \leadsto \overline{N'} \leadsto \overline{N''}$, we say that $\overline{N}$ is \brown{closer} to $\overline{N'}$ than to $\overline{N''}$.
 
\begin{lemm}\label{exist_path}
$[\alpha_k,\dots,\alpha_1] \leadsto [\beta_k,\dots,\beta_1]$ if and only if 
$\sum\limits_{j=i}^k \alpha_j \geq \sum\limits_{j=i}^k \beta_j$ for all $1 \leq i \leq k$.
\end{lemm}

\begin{proof}
In the proof of Lemma \ref{8}, $[\alpha_k,\dots,\alpha_1] \leadsto [\beta_k,\dots,\beta_1]$ if and only if there exist non-negative integers $t_i \geq 0$ ($i=1,\dots,k$) satisfying
\begin{align*}
\alpha_k - t_k &= \beta_k, \\
\alpha_{k-1} + t_k - t_{k-1} &= \beta_{k-1}, \\
\alpha_{k-2} + t_{k-1} - t_{k-2} &= \beta_{k-2}, \\
&\vdots \\
\alpha_2 + t_3 - t_2 &= \beta_2, \\
\alpha_1 + t_2 - t_1 &= \beta_1.
\end{align*}
This system has a solution in non-negative integers if and only if 
$t_i = \sum\limits_{j=i}^k \alpha_j - \sum\limits_{j=i}^k \beta_j \geq 0$ for all $i = 1,\dots,k$.
\end{proof}

\begin{lemm}\label{inequality_Delta}
If $[\alpha_k,\dots,\alpha_1] \leadsto [\beta_k,\dots,\beta_1]$, then 
$\Delta^{[\alpha_k,\dots,\alpha_1]} \geq \Delta^{[\beta_k,\dots,\beta_1]}$.
\end{lemm}

\begin{proof}
By Lemma \ref{8}, we have
$$
f([\alpha_k,\dots,\alpha_1]) \geq \sum_{j=1}^k j(\alpha_j - \beta_j) + f([\beta_k,\dots,\beta_1]).
$$
This implies the desired conclusion.
\end{proof}

The next lemma gives the formula for $f(n)$ for certain special numbers $n$.
\begin{lemm}\label{length_Vakil}
For $n = 2^{a+1} k$ where $k \in [a, 2a+1]$, we have $f(n) = k + \frac{a(a+1)}{2}.$
\end{lemm}

\begin{proof}\hfill
\begin{itemize}
    \item We first show that the minimum number $n$ such that $f(n)=l$ is $n=2^{a+1}\Bigl(l - \frac{a(a+1)}{2}\Bigr)$, where $\frac{a(a+3)}{2} \leq l < \frac{(a+1)(a+4)}{2}$. Indeed, the minimum number $n$ such that $f(n)=l$ is $\sum\limits_{i=0}^{l-1} F(i)$, where $F(s) = \#\{n \mid g(n)=s\}$. By Theorem \ref{3}, in the frequency table of $g(n)$, the value $2^a$ appears $a+1$ times and $F(0)=2^1$. Hence, $F(s)=2^{a+1}$ if and only if 
    $$
    2 + 3 + \dots + (a+1) \leq s < 2 + 3 + \dots + (a+2),
    $$
    or equivalently,
    $$
    \frac{a(a+3)}{2} \leq s < \frac{(a+1)(a+4)}{2}.
    $$
    
    Now, for $s \in \Bigl[\frac{a(a+3)}{2}, \frac{(a+1)(a+4)}{2}\Bigr)$ we have
    $$
    \sum_{i=0}^{s-1} F(i) = \sum_{j=1}^a 2^j (j+1) + 2^{a+1} \Bigl(s - \frac{a(a+3)}{2}\Bigr) = 2^{a+1} \Bigl(s - \frac{a(a+1)}{2}\Bigr).
    $$
    Hence $n = 2^{a+1}\Bigl(l - \frac{a(a+1)}{2}\Bigr)$, where $\frac{a(a+3)}{2} \leq l < \frac{(a+1)(a+4)}{2}$.
    
    \item We have $n = 2^{a+1} k = 2^{a+1} \Bigl(l - \frac{a(a+1)}{2}\Bigr)$, where $l = k + \frac{a(a+1)}{2} \in \Bigl[\frac{a(a+3)}{2}, \frac{(a+1)(a+4)}{2}\Bigr)$ since $k \in [a, 2a+1]$. Therefore $f(n) = l = k + \frac{a(a+1)}{2}$.
\end{itemize}
\end{proof}

For a real number $r$, $z \approx r$ means that $z$ is the largest integer such that $z \leq r$, and $z \overset{+}{\approx} r$ means that $z$ is the largest integer such that $z < r$. The following is a corollary of Lemma \ref{for_Vakil}.

\begin{coro}\label{f_n_leq_log2i}
For $i \geq 2$ and $n < \log_2 i$, we have
$$
f([n,0_1,\dots,0_{i-1}]) = (2^n-1)2^k + \frac{(i-k)(i-k-1)}{2},
$$
where $k \geq 1$ is such that $(2^n-1)2^{k-1} + k + 1 \approx i$, or equivalently, $i \approx (2^n-1)2^k + k + 1$.
\end{coro}
\begin{proof}
For $n < \log_2 i$, $k \geq 1$ is well-defined because $2^n < i$ implies $(2^n-1)2^0 + 2 \leq i$.  

Set $a = i - k - 1$. We will prove that
$$
2^{a+1} a \leq 2^i (2^n-1) < 2^{a+2} (a+1).
$$
Indeed,
\begin{equation}\label{a}
2^{a+1} a \leq 2^i (2^n-1) \quad \Leftrightarrow \quad i \leq (2^n-1)2^k + k + 1,
\end{equation}
\begin{equation}\label{b}
2^i (2^n-1) < 2^{a+2} (a+1) \quad \Leftrightarrow \quad (2^n-1)2^{k-1} + k < i.
\end{equation}
Both (\ref{a}) and (\ref{b}) follow from our assumption that $(2^n-1)2^{k-1} + k + 1 \approx i$ (equivalently, $i \approx (2^n-1)2^k + k + 1$).  

By Lemma \ref{for_Vakil}, we have
$$
f\bigl(2^i (2^n-1)\bigr) = \frac{2^i (2^n-1)}{2^{a+1}} + \frac{a(a+1)}{2} = (2^n-1)2^k + \frac{(i-k)(i-k-1)}{2}.
$$
Since $\overline{2^i (2^n-1)} = [n,0_1,\dots,0_{i-1}]$, the conclusion follows.
\end{proof}

\begin{ex}
In Example \ref{6}, we see that $[3,2,0] \leadsto [2,0,0] \leadsto [1,0,1]$; hence, $[3,2,0]$ is closer to $[2,0,0]$ than to $[1,0,1]$. 
Lemmas \ref{8}, \ref{exist_path}, and \ref{inequality_Delta} can be seen through the path $[3,2,0] \leadsto [1,0,1]$. Indeed,
\begin{itemize}
    \item the path from $[3,2,0]$ to $[1,0,1]$ has length $9$,
    \item we have the inequalities $3>1$, $3+2>1+0$, and $3+2+0>1+0+1$,
    \item we have 
    $$\Delta^{[3,2,0]}=l([3,2,0])-S([3,2,0])= 13 - 13 = 0,$$ 
    $$\Delta^{[1,0,1]}=l([1,0,1])-S([1,0,1])= 4 - 4 = 0.$$
\end{itemize}
Lemma \ref{length_Vakil} and Corollary \ref{f_n_leq_log2i} can be seen through $[1,0,0]$. Indeed,
\begin{itemize}
    \item $l([1,0,0])=3$,
    \item the pair $(a,k)$ in Lemma \ref{length_Vakil} is $(1,2)$,
    \item the numbers $i,n,k$ in Corollary \ref{f_n_leq_log2i} are $3,1,1$, respectively.
\end{itemize}
\end{ex}

\subsection{Canonical path to Vakil numbers}

\begin{defi}
We call a non-negative integer $N$ a \brown{Vakil number} if $\overline{N}$ has the form $2^{a+1}k$ $(a \leq k \leq 2a+1$, $k \in \mathbb{Z}_{>0})$, and a \brown{$d$-Vakil number} if its dimension is $d$. We call $(a,k)$ the \brown{Vakil pair associated with $N$}, and write it as $V(N)$. A pair $(a,k)$ is a \brown{Vakil pair} if and only if $k \in [a,2a+1]$.
\end{defi}

Suppose that 
\[
\overline{N}=\underbrace{1*\dots *}_{\alpha}\,\underbrace{*\dots*10\dots0}_{\beta}\,_{(2)},
\]
where \(2^{l(\beta)}(l(\beta)-1)\leq \overline{N} < 2^{l(\beta)+1}l(\beta)\).
Then the number \(N\) is a Vakil number if and only if \(\beta = 0\dots 0\), and it is not a Vakil number if and only if \(\beta = *\dots*10\dots0\).
We already know the formula for \(f(N)\) when \(N\) is a Vakil number by Lemma~\ref{for_Vakil}.
We will prove that, after finitely many steps along the canonical path, a number \(N\) which is not a Vakil number will meet the first Vakil number \(N'\) of the same dimension.
Hence, by Remark~\ref{f_by_cano} and Lemma~\ref{8}, we have
\[
\Delta^N= \Delta^{\overline{N}}=\Delta^{N'}.
\]
Here, we can compute \(\Delta^{N'}\) explicitly by Lemma~\ref{for_Vakil}.

\begin{theo}\label{18} If $N$ is not a Vakil number, then some first steps in the canonical path of $T_{\overline{N}}$ are 
\begin{align*}
\overline{N}=&N_0=1*\dots*\underbrace{*\dots*1\overbrace{000\dots00}^{b \text{ times} }}_{\text{length }l(\beta)}\,_{(2)}\\
\longrightarrow \quad& N_1= 1*\dots*\underbrace{*\dots*0100\dots00}_{\text{length } l(\beta)}\,_{(2)}\\
\longrightarrow \quad& N_2=1*\dots*\underbrace{*\dots*0010\dots00}_{\text{length } l(\beta)}\,_{(2)}\\
\longrightarrow \quad& N_3=1*\dots*\underbrace{*\dots*0001\dots00}_{\text{length } l(\beta)}\,_{(2)}\\
 \quad& \dots\\
\longrightarrow \quad& N_{b-1}= 1*\dots*\underbrace{*\dots*0000\dots10}_{\text{length } l(\beta)}\,_{(2)}\\
\longrightarrow \quad& N_b= 1*\dots*\underbrace{*\dots*0000\dots 0\cancel{1}}_{\text{length } l(\beta)}\,_{(2)}
\end{align*}
with $l(\beta(N_k))=l(\beta)$ for $0\leq k \leq b-1$ and $l(\beta)-1 \leq l(\beta(N_b))\leq l(\beta)$.
\end{theo}
\begin{proof}
Let $\alpha_k,\beta_k$ be the $\alpha,\beta$ parts of $N_k$, and let $s_k$ be $s(N_k)$. We have $N_k=\alpha_k\beta_k$. 
\begin{itemize}
\item[1.]
First, we prove that the edge
\[
\underbrace{1*\dots*}_\alpha\underbrace{*\dots*1\overbrace{000\dots00}^{b\geq 2}}_\beta\,_{(2)} 
\longrightarrow \underbrace{1*\dots*}_{\alpha'=\alpha}\underbrace{*\dots*0100\dots00}_{\beta'}\,_{(2)}
\]
is canonical with canonical indices $s=b-1$ and $s'=s-1$. Indeed, $s=b-1 \geq 1$ by definition. We have $l(\beta')=l(\beta)$. Indeed, since $2^s\leq \beta<2^{l(\beta)}$, we have
\[
\frac{\alpha'\beta'}{2^{l(\beta)}}= \frac{\alpha\beta-2^s}{2^{l(\beta)}}=\alpha + \frac{\beta-2^s}{2^{l(\beta)}}=\alpha +\epsilon
\]
for some $\epsilon \in [0,1)$. Since $l(\beta)-1\leq \alpha < 2l(\beta)$, we have
\[
l(\beta)-1 \leq \frac{\alpha'\beta'}{2^{l(\beta)}}<2l(\beta).
\]
This implies that $l(\beta')=l(\beta)$. Since $\beta'=*\dots*010\dots0$, we have $s'=b-2=s-1$.
\item[2.]
Now, applying 1. to $N_0=\alpha_0\beta_0$, we obtain the canonical edge $N_0\rightarrow N_1$. Similarly, for $N_1$, and so on, we obtain the first steps in the canonical path
\[
N_0\rightarrow N_1\rightarrow \dots \rightarrow N_{b-1}.
\]
Because $s_{b-1}=0$, we obtain the edge
\[
N_{b-1}\rightarrow N_b = \frac{N_0}{2}-2^{b-1}.
\]
\item[3.]
To prove that $l(\beta)-1 \leq l(\beta(N_b))\leq l(\beta)$, it suffices to show that
\[
2^{l(\beta)-1}(l(\beta)-2) \leq \frac{N_0}{2}-2^{b-1} < 2^{l(\beta)+1}l(\beta).
\]
This holds because
\[
2^{l(\beta)}(l(\beta)-1)\leq N_0 < 2^{l(\beta)+1}l(\beta).
\]
\end{itemize}
\end{proof}

We see that $l(\beta_k)=l(\beta)$ for all $0\leq k<b$. Hence, $N_k$ cannot be a Vakil number because $N_k/2^{l(\beta)} \notin \mathbb{Z}$. However, $N_b$ can be a Vakil number in some cases. Moreover, if $N_0$ and $N_c$ ($c\geq b$) have the same dimension and $N_0$ reaches $N_c$ through a canonical path, then $\Delta^{N_0}=\Delta^{N_c}$ by Remark~\ref{f_by_cano} and Lemma~\ref{8}. In bracket form, if
\[
N_0 =[\alpha_k,\dots, \alpha_j, 0_1,\dots,0_{j-1}] \qquad (\alpha_j,j\geq 1),
\]
then
\[
N_b=[\alpha_k,\dots,\alpha_j-1,0_1,\dots,0_{j-1}].
\]
Thus, we can compute $\Delta^N$ by the \brown{reduction process}:
\begin{enumerate}\label{reduction}
\item[1.] Write $\overline{N}=[\alpha_k,\dots,\alpha_1]$.
\item[2.] Reduce each $\alpha_i\geq 1$ in the following way, from top to bottom and left to right. If $\alpha_i=0$, we do not perform the reduction $\alpha_i \rightarrow \alpha_i-1$. Here we use $\xrightarrow[]{(i)}$ for the $i$-th reduction step.
\begin{align*}
\alpha_1\xrightarrow[\quad\quad\quad]{(1)} \alpha_1-1\xrightarrow[\quad\quad\quad]{(2)} \dots\xrightarrow[\quad\quad\quad]{(\alpha_1)} 0\\
\alpha_2\xrightarrow[\quad\quad\quad]{(\alpha_1+1)} \alpha_2-1\xrightarrow[\quad\quad\quad]{(\alpha_1+2)}\dots\xrightarrow[\quad\quad\quad]{(\alpha_1+\alpha_2)} 0\\
\dots\\
\alpha_k\xrightarrow[\quad\quad\quad]{} \alpha_k-1\xrightarrow[\quad\quad\quad]{}\dots\xrightarrow[\quad\quad\quad]{} 0
\end{align*}
\item[3.] The process finishes when we reach the first Vakil number $N'$ of the same dimension. The existence of $N'$ is guaranteed by Lemma~\ref{21} below.
\item[4.] Then we have $\Delta^N = \Delta^{\overline{N}} = \Delta^{N'}$.
\end{enumerate}
The number $N'$ above is the $k$-Vakil number that is closest to $\overline{N}$.

\begin{lemm}\label{21}
Let $x=\lceil \log_2 d \rceil$. 
The number $2^d = [1,0_1,\dots,0_{d-1}]$ is a $d$-Vakil number with Vakil pair
\[
V(2^d) = \begin{cases}
(d-x,2^{x-1}) & \text{if } 2^{x-1}+x\geq d,\\
(d-x-1,2^x) & \text{if } 2^{x-1}+x<d.
\end{cases}
\]
\end{lemm}

\begin{proof}
Since $x=\lceil \log_2 d \rceil$, we have $d-x-1\leq 2^x$ and $2^{x-1}\leq 2(d-x)+1$. If $2^{x-1}+x\geq d$, then $2^{x-1}\geq d-x$, and hence $V(2^d)=(d-x,2^{x-1})$. If $2^{x-1}+x<d$, then $2^x\leq 2(d-x-1)+1$, and hence $V(2^d)=(d-x-1,2^x)$.
\end{proof}

\begin{rema}
The reduction process counts canonical edges from $\overline{N}$ to the first Vakil number $N'$, but not the whole canonical path of $T_{\overline{N}}$. The reason is that Theorem~\ref{18} does not apply to Vakil numbers. Hence, if $N''$ is the second Vakil number we meet in the reduction process, we have $\Delta^N \geq \Delta^{N''}$, but this does not guarantee equality. For example, as we will see later in Table~\ref{tab_53}, there are some Vakil numbers that can reach another Vakil number (as we know by Lemma~\ref{exist_path}), but the value of $\Delta$ decreases strictly. On the other hand, as we will see in Example~\ref{ex_reduction} below or in Theorem~\ref{23} later, equality still often appears.     
\end{rema}

\begin{ex}\label{ex_reduction}
In Examples~\ref{5} and~\ref{6}, the number $473$ is not a Vakil number, because $\overline{473}=236$ does not have the form $2^{a+1}k$ $(a\leq k \leq 2a+1,\, k \in \mathbb{Z}_{>0})$. Its binary form with $\alpha,\beta$ parts is $\underbrace{111}_{\alpha} \underbrace{01100}_{\beta}\,_{(2)}$. We see that $\beta \ne 0\dots 0$. The first steps mentioned in Theorem~\ref{18} in the canonical path of $T_{\overline{236}}$ are (in integer, binary, and bracket form):

\begin{align*}
&N_0 = 236=111\underbrace{011\overbrace{00}^{2 \text{ times} }}_{\text{length }5}\,_{(2)} = [3,2,0]\\
\longrightarrow \quad& N_1 = 234 =111\underbrace{01010}_{\text{length } 5}\,_{(2)} = [3,1,1]\\
\longrightarrow \quad& N_2=116=111\underbrace{0100\cancel{1}}_{\text{length } 5}\,_{(2)} = [3,1,0].
\end{align*}
We have $\beta(N_0) = 01100$, $\beta(N_1) = 01010$, and $\beta(N_2) = 0100$. The reduction process to compute $\Delta^{473}$ is:
\begin{itemize}
    \item[1.] Write $236 = [3,2,0]$.
    \item[2.] Make reductions on the bracket $[3,2,0]$:
    \begin{align*}
    \alpha_1&=0,\\
    \alpha_2&=2 \xrightarrow[]{(1)} 1 \xrightarrow[]{(2)} 0,\\
    \alpha_3&=3 \xrightarrow[]{(3)} 2 \xrightarrow[]{(4)} 1 \xrightarrow[]{(5)} 0 .
    \end{align*}
    Reading the steps from top to bottom and left to right, we obtain the sequence of brackets
    \[
    [3,2,0] \xrightarrow[]{(1)} [3,1,0] \xrightarrow[]{(2)} [3,0,0] \xrightarrow[]{(3)} [2,0,0] \xrightarrow[]{(4)} [1,0,0] \xrightarrow[]{(5)} [0,0,0] = [0].
    \]
    We see that $N_0 \xrightarrow[]{(1)} N_2$ is $[3,2,0] \xrightarrow{(1)} [3,1,0]$. 
    \item[3.] Now, $[1,0,0]$ is a $3$-Vakil number with Vakil pair $(1,2)$ by Lemma~\ref{21}. This guarantees that $[3,2,0]$ can reach a first Vakil number of dimension $3$. In this case, the first Vakil number is $[2,0,0]=24$ with Vakil pair $(2,3)$.
    \item[4.] By Remark~\ref{f_by_cano} and Lemma~\ref{8}, we have $\Delta^{473}=\Delta^{236}=\Delta^{24}$. By Lemma~\ref{for_Vakil}, we have
    \[
    \Delta^{24} = f(24) - S([2,0,0]) = 6-6=0.
    \]
\end{itemize}
The second Vakil number we meet in the reduction process is $[1,0,0]=8$. By Lemma~\ref{for_Vakil}, we have
\[
\Delta^{8} = f(8)-S([1,0,0])=3-3=0.
\]
\end{ex}

\subsection{Values of $\Delta$ for $k$-Vakil numbers}

\begin{lemm}\label{19}
Set $S'([\alpha_k, \dots,\alpha_1])=\sum\limits_{j=1}^k \alpha_j$. If $N=[\alpha_k,\dots,\alpha_2,\alpha_1]\alpha_0$, then
\begin{enumerate}
\item[a)] $2^{a+1}N = [\alpha_k,\dots,\alpha_0,0_1,\dots,0_a]$ and
\[
S(2^{a+1}N)= S([\alpha_k,\dots,\alpha_0]) + S'([\alpha_k,\dots,\alpha_0])\,a,
\]
\item[b)] $N+1=\begin{cases}
[\alpha_k,\dots,\alpha_1+1,0_1,\dots,0_{\alpha_0-1}] & \text{if } \alpha_0\geq 1, \\[2pt]
[\alpha_k,\dots,\alpha_2](\alpha_1+1) & \text{if } \alpha_0=0.
\end{cases}$
\end{enumerate}
\end{lemm}

\begin{proof}
This follows by a simple computation.
\end{proof}

\begin{prop}\label{20}
The numbers $2^{a+1}N$ and $2^{a'+1}(N+1)$ have the same positive dimension if and only if $a'=a-t_N+1_N$, where $1_N$ is equal to $1$ if $N+1$ is not a power of $2$, and it is equal to $0$ otherwise.
\end{prop}

\begin{proof}
Suppose that $N=[\alpha_k, \dots,\alpha_1]\alpha_0$. We consider the following two cases.
\begin{itemize}
\item[1.] If $N$ is even ($\alpha_0=0$), then by Lemma~\ref{19} we have
\[
2^{a'+1}(N+1)=[\alpha_k, \dots,\alpha_2,\alpha_1+1,0_1,\dots,0_{a'}].
\]
There exists $\alpha_j>0$ for some $j\geq 1$ because $2^{a+1}N$ has positive dimension. Hence, $2^{a+1}N$ and $2^{a'+1}(N+1)$ have the same dimension if and only if $a'=a+1$.
\item[2.] If $N$ is odd ($\alpha_0>0$), we have
\[
2^{a'+1}(N+1)=[\alpha_k,\dots,\alpha_1+1,0_1,\dots,0_{\alpha_0-1},0,0_1,\dots,0_{a'}].
\]
If $N=[0]\alpha_0$, then $2^{a'+1}(N+1)$ and $2^{a+1}N$ have the same dimension if and only if $a'=a-\alpha_0$. Otherwise, we have $a'=a-\alpha_0+1$.
\end{itemize}
Therefore, $a'=a-t_N+1_N$.
\end{proof}

\begin{prop}\label{22}
Suppose that $N$ and $N'$ are $d$-Vakil numbers with $V(N)=(a,k)$ and $V(N')=(a',k+1)$. If $k+1$ is not a power of $2$, then
\[
\Delta^N-\Delta^{N'}=\frac{t_k(t_k-1)}{2}.
\]
\end{prop}

\begin{proof}
Suppose that $k=[\alpha_l,\dots,\alpha_1]\alpha_0$. By Lemma~\ref{19}, we have
\begin{align*}
N&=[\alpha_l,\dots,\alpha_1,\alpha_0,0_1,\dots,0_{a}],\\
N'&=[\alpha_l,\dots,\alpha_1+1,0_1,\dots,0_{a+1}].
\end{align*}
By Lemma~\ref{for_Vakil}, we have 
\begin{align*}
\Delta^N&=k+\frac{a(a+1)}{2}-S([\alpha_l,\dots,\alpha_0])-S'([\alpha_l,\dots,\alpha_0])a,\\
\Delta^{N'}&=(k+1)+\frac{a'(a'+1)}{2}-S([\alpha_l,\dots,\alpha_1+1,0])-S'([\alpha_l,\dots,\alpha_1+1,0])a.
\end{align*}
Since $k+1$ is not a power of $2$, Proposition~\ref{20} gives $a'=a-t_k+1$. Hence,
\[
\Delta^N-\Delta^{N'}=\frac{t_k(t_k-1)}{2}.
\]
\end{proof}

\begin{theo}\label{23}
For $0\leq i\leq 4$, set $N_i=2^{a_i+1}k_i$, where $(a_i,k_i)$ are pairs of non-negative integers such that $k_{i+1}=k_i+1$, $a_{i+1}=a_i-t_{k_i}+1_{k_i}$, $4|k_0$, and $k_4$ is not a power of $2$. We have
\begin{enumerate}
\item[a)] $a_1=a_2=a_0+1$, $a_3=a_0+2$, and $a_4=a_0-t_{k_0/4}+1$.
\item[b)] If $N_0$ is a Vakil number with Vakil pair $(a_0,k_0)$, then $N_i$ are also Vakil numbers with Vakil pairs $(a_i,k_i)$ for $i=1,2,3$, and
\[
\Delta^{N_0} =\Delta^{N_1}=\Delta^{N_2}=\Delta^{N_3}.
\]
In addition, if $N_4$ is a Vakil number, then
\begin{equation*}\label{Delta_N4_by_N0}
    \Delta^{N_4}=\Delta^{N_0} + \begin{cases}
    \dfrac{(t_{k_0/4}+1)(t_{k_0/4}+2)}{2} & \text{if } V(N_4)=(a_4,k_4),\\[6pt]
    \dfrac{(t_{k_0/4}+1)(t_{k_0/4}+2)}{2}+(a_4+1)-\dfrac{k_4}{2} & \text{if } V(N_4)=(a_4+1,k_4/2).
    \end{cases}
\end{equation*}
\item[c)] If $N_0$ is a Vakil number with Vakil pair $(a_0+1,k_0/2)$, then $N_2$ is also a Vakil number with Vakil pair $(a_2+1,k_2/2)$ and
\[
\Delta^{N_0}=\Delta^{N_2}.
\]
In addition, if $N_1$ and $N_3$ are Vakil numbers with Vakil pairs $(a_i,k_i)$, then
\[
\Delta^{N_1}=\Delta^{N_3}=\Delta^{N_0}.
\]
\end{enumerate}
\end{theo}

\begin{proof}\hfill
\begin{itemize}
\item[a)] This follows by a simple computation, with the remark that $t_{k_3}=t_{k_0/4}+2$.

\item[b)] By Proposition \ref{20}, the numbers $N_i$ have the same dimension. Since $N_0$ is a Vakil number, the pair $(a_0,k_0)$ satisfies the inequality $a_0 \leq k_0 \leq 2a_0+1$. This implies that the same holds for the pairs $(a_i,k_i)$ for $i=1,2,3$. Hence, $N_i$ are Vakil numbers with Vakil pairs $(a_i,k_i)$ for $i=1,2,3$, and
\[
\Delta^{N_0}=\Delta^{N_1}=\Delta^{N_2}=\Delta^{N_3}.
\]
If $N_4$ is a Vakil number with pair $(a_4,k_4)$, then
\[
\Delta^{N_4}-\Delta^{N_3}=\frac{(t_{k_0/4}+1)(t_{k_0/4}+2)}{2}.
\]
As in the proof of Proposition \ref{22}, when we replace the pair $(a_4,k_4)$ by $(a_4+1,k_4/2)$, the value of $\Delta^{N_4}$ changes by
\[
\frac{k_4}{2}+\frac{(a_4+1)(a_4+2)}{2}-k_4-\frac{a_4(a_4+1)}{2}
=(a_4+1)-\frac{k_4}{2}.
\]

\item[c)] The first conclusion follows from a simple computation. For the second conclusion, if $N_i$ is a Vakil number with pair $(a_i,k_i)$ for $i=1,3$, then $k_0=2(a_0+1)$. Replacing the pair $(a_0,k_0)$ by $(a_0+1,k_0/2)$ changes the value of $\Delta^{N_0}$ by
\[
(a_0+1)-\frac{k_0}{2}=0.
\]
Thus, by part b), we have $\Delta^{N_0}=\Delta^{N_1}$. By the same argument, we obtain $\Delta^{N_2}=\Delta^{N_3}$.
\end{itemize}
\end{proof}

\begin{prop}\label{25}
If $N$ is a $d$-Vakil number with $V(N)=(a,k)$ and $V(2^d)=(a_0,k_0)$, then $a\geq a_0$.
\end{prop}

\begin{proof}
We first show that if
\[
k+\frac{a(a+1)}{2} \leq h+\frac{b(b+1)}{2},
\]
with $k\in [a,2a+1]$ and $h\in [b,2b+1]$, then $a\leq b$. Indeed, if $b\leq a-1$, then we obtain a contradiction:
\[
2b+1+\frac{b(b+1)}{2}\leq a+\frac{a(a+1)}{2}-1
< k+\frac{a(a+1)}{2}
\leq h+\frac{b(b+1)}{2}.
\]
Since $N \leadsto 2^d$, we have $f(N)\geq f(2^d)$. By Lemma \ref{for_Vakil}, we conclude that $a\geq a_0$.
\end{proof}

\begin{theo}\label{26}
For $d>4$, let $x=\lceil \log_2 d \rceil$. 
\begin{enumerate}
\item If $2^{x-1}+x<d$, then
\[
\{\text{$d$-Vakil numbers}\}
\subset
\{\text{$2^{a+1}k$ of dimension $d$ with $k \in [2^x,2^{x+1})$}\}.
\]
For those pairs $(a,k)$, if $(a,k)$ is not a Vakil pair and $k$ is even, then
$V(2^{a+1}k)=(a+1,k/2)$. If $(a,k)$ is not a Vakil pair and $k$ is odd, then
$2^{a+1}k$ is not a Vakil number.
\item If $2^{x-1}+x\geq d$, then
\[
\{\text{$d$-Vakil numbers}\}
=
\{\text{$2^{a+1}k$ of dimension $d$ with $k \in [2^{x-1},2^x)$}\}.
\]
\end{enumerate}
\end{theo}

\begin{proof}
We have $x=\lceil\log_2d \rceil$ if and only if $2^{x-1}<d\leq 2^x$.
\begin{enumerate}
\item[1.]\begin{itemize}
\item
If $2^{x-1}+x<d$, then $V(2^d)=(d-x-1,2^x)$ by Lemma \ref{21}. Suppose that $N$ is a $d$-Vakil number with $V(N)=(a',k')$. Then $a'\geq d-x-1$ by Proposition \ref{25}. Hence $k' \in [2^{x-1},2^{x+1})$ because
\begin{align*}
k'&\geq a' \geq d-x-1\geq 2^{x-1},\\
k'&\leq 2a'+1\leq 2(d-1)+1 \leq 2(2^x-1)+1<2^{x+1}.
\end{align*}
If $k' \in [2^{x-1},2^x)$, then $2k' \in [2^x,2^{x+1})$.

\item
If $(a,k)$ is not a Vakil pair and $k$ is even, then $k\geq 2a+2$ because the rate of increase from $2^x$ to $k$ is faster than the rate of increase from $d-x-1$ to $a$ (by Proposition \ref{20}), and $2^x\geq d-x-1$. Hence $k/2\geq a+1$. Therefore, if $(a+1,k/2)$ is not a Vakil pair, we must have $k/2>2(a+1)+1$, or equivalently $k>4a+6$. This implies
\[
2^{x+1} > k>4a+6\geq 4(d-x-1)+6 \geq 4(d-x)+2.
\]
Hence,
\[
2^{x}>2(d-x)+1.
\]
Thus,
\[
2^{x}\geq 2(d-x+1) \Longleftrightarrow 2^{x-1}\geq d-x+1.
\]
This yields a contradiction:
\[
2^{x-1}+x\geq d+1>d.
\]
Therefore, $V(2^{a+1}k)=(a+1,k/2)$.

\item
If $k$ is odd and $V(2^{a+1}k)=(a',k')$, then $k'>k$, hence $a'<a$. Thus
\[
k'>k>2a+1>2a'+1,
\]
which is a contradiction.
\end{itemize}

\item[2.]
\begin{itemize}

\item If $2^{x-1}+x\geq d$, then $V(2^d)=(d-x,2^{x-1})$ by Lemma \ref{21}. 

\item First, we prove that all pairs $(a,k)$ such that $2^{a+1}k$ has dimension $d$ and $k\in [2^{x-1},2^x)$ are Vakil pairs. Indeed, as in the first part, the rate of increase from $2^{x-1}$ to $k$ is faster than the rate of increase from $d-x$ to $a$, hence $k\geq a$. We must show that $k\leq 2a+1$. 

\begin{itemize}
    \item Put $d=2^{x-1}+\epsilon_1$ for some $\epsilon_1 \in (0,x]$, since $2^{x-1}<d\leq 2^{x-1}+x$. Put $a=d-x+\epsilon_2$ for some $\epsilon_2\geq 0$. Put $k=2^{x-1}+\epsilon_3$ for some $\epsilon_3 \in [0,2^{x-1})$.
    
    \item We have
\begin{equation}\label{(*)}
    k\leq 2a+1 \Longleftrightarrow \frac{\epsilon_3-1}{2}-\epsilon_2 \leq 2^{x-2}-x+\epsilon_1.
\end{equation}
Since $\epsilon_1-x \geq 1-x$, to prove \eqref{(*)}, it suffices to prove
\begin{equation} \label{(**)}
    \frac{\epsilon_3-1}{2}-\epsilon_2 \leq 2^{x-2}-x+1.
\end{equation}
By part b) of Theorem \ref{23}, it suffices to prove that $(a,k)$ are Vakil pairs for $4|k$. Since $d>4$, we have $4|2^{x-1}$. Hence, we prove \eqref{(**)} for $4|\epsilon_3$.
    
\item We observe that $\frac{\epsilon_3-1}{2}-\epsilon_2$ is increasing as $\epsilon_3$ increases with $4| \epsilon_3$. Indeed,
\[
\frac{\epsilon_3-1}{2}-\epsilon_2 < \frac{(\epsilon_3+4)-1}{2}-\epsilon_2'
\Longleftrightarrow \epsilon_2' <\epsilon_2+2.
\]
The latter inequality holds by part a) of Theorem \ref{23}, where $\epsilon_2$ and $\epsilon_2'$ correspond to $a_0$ and $a_4$, respectively.
    
\item The maximum value of $\epsilon_3$ such that $4| \epsilon_3$ is $2^{x-1}-4$, and the corresponding value of $\epsilon_2$ is $x-3$. To compute $\epsilon_2$, we consider the pair $(d-x+\epsilon_2,2^x-4)$ as $(a_0,k_0)$ and the pair $(d-x-1,2^x)$ as $(a_4,k_4)$ in Theorem \ref{23}, with the remark that $k_4$ is a power of $2$. Hence, $a_4=a_0-t_{k_0/4}$, which gives
\[
\epsilon_2=-1+t_{\frac{2^x-4}{4}}=x-3.
\]
Therefore, for $4| \epsilon_3$ and $\epsilon_3 \in [0,2^{x-1})$, we have
\[
\frac{\epsilon_3-1}{2}-\epsilon_2 \leq \frac{2^{x-1}-5}{2}-(x-3)
=2^{x-2}-x+\frac{1}{2}
< 2^{x-2}-x+1.
\]
    
\item Thus, \eqref{(**)} is proved, which implies $k \leq 2a+1$ when $4| k$. 
\end{itemize}

Hence, $(a,k)$ are Vakil pairs for all $k$ such that $4| k$ and $k\in [2^{x-1},2^x)$. By part b) of Theorem \ref{23}, this holds for all $k\in [2^{x-1},2^x)$.

\item Second, we prove that if $(a,k)$ is a Vakil pair, then $k\in [2^{x-1},2^x)$. We have $k\in [2^{x-2},2^{x+1})$ because
$$
2^{x+1}>k\geq a \geq d-x \geq 2^{x-1}-x \geq 2^{x-2}.
$$
Thus,
\[
k \in [2^{x-2},2^{x-1})\cup [2^{x-1},2^x)\cup[2^{x},2^{x+1}).
\]
If $k \in [2^{x-2},2^{x-1})$, then $2k \in [2^{x-1},2^x)$. As we know, $(a-1,2k)$ is a Vakil pair, hence $k\notin [2^{x-2},2^{x-1})$.

For $k \in [2^{x},2^{x+1})$, we will prove that the inequality \eqref{(*)} does not hold with $\epsilon_3 \geq 2^{x-1}$ (in this case, $\epsilon_2\geq -1$), that is,
\[
\frac{\epsilon_3-1}{2}-\epsilon_2 > 2^{x-2}-x+\epsilon_1.
\]
Since $\epsilon_1-x\leq 0$, it suffices to prove that
\[
\frac{\epsilon_3-1}{2}-\epsilon_2>2^{x-2}.
\]
It is enough to prove this inequality for $4| \epsilon_3$, because if
$\frac{\epsilon_3-1}{2}-\epsilon_2 >2^{x-2}$, then
\begin{align*}
\frac{(\epsilon_3+1)-1}{2}-(\epsilon_2+1)
&=\frac{\epsilon_3-1}{2}-\epsilon_2-\frac{1}{2}>2^{x-2},\\
\frac{(\epsilon_3+2)-1}{2}-(\epsilon_2+1)
&=\frac{\epsilon_3-1}{2}-\epsilon_2>2^{x-2},\\
\frac{(\epsilon_3+3)-1}{2}-(\epsilon_2+2)
&=\frac{\epsilon_3-1}{2}-\epsilon_2-\frac{1}{2}>2^{x-2}.
\end{align*}
These imply that if $(a,k)$ is not a Vakil pair for $4| k$, then $(a+1,k+1)$, $(a+1,k+2)$, and $(a+2,k+3)$ are not Vakil pairs.

Since $\frac{\epsilon_3-1}{2}-\epsilon_2$ is increasing as $\epsilon_3$ increases with $4| \epsilon_3$, and since $\epsilon_2=-1$ when $\epsilon_3=2^{x-1}$, we have
\[
\frac{\epsilon_3-1}{2}-\epsilon_2
\geq \frac{2^{x-1}-1}{2}+1
=2^{x-2}+\frac{1}{2}
>2^{x-2}.
\]
Thus, $k\notin [2^x,2^{x+1})$. In conclusion, we have
\[
\{\text{$d$-Vakil numbers}\}
=
\{\text{$2^{a+1}k$ of dimension $d$ and $k \in [2^{x-1},2^x)$}\}.
\]
\end{itemize}
\end{enumerate}
\end{proof}

\begin{ex}\label{combine_ex}
Lemma \ref{19}, Propositions \ref{20}, \ref{22}, \ref{25}, and Theorem \ref{23} a), b) can be seen through the pairs
$(a_0,k_0)= (46,64)$, $(a_1,k_1)=(47,65)$, $(a_2,k_2)=(47,66)$, $(a_3,k_3)=(48,67)$, $(a_4,k_4)=(47,68)$.
We have 
\begin{align*}
    k_0&=[1,0,0,0,0,0]\quad &N_0=[1,0,0,0,0,0,\dots,0], \\
    k_1&=[1,0,0,0,0]1\quad &N_1=[1,0,0,0,0,1,\dots,0],\\
    k_2&=[1,0,0,0,1]\quad &N_2=[1,0,0,0,1,0,\dots,0],\\
    k_3&=[1,0,0,0]2\quad &N_3=[1,0,0,0,2,0,\dots,0],\\
    k_4&=[1,0,0,1,0]\quad &N_4=[1,0,0,1,0,0,\dots,0].
\end{align*}
Here, the dimension of $N_i$ is $53$. By Lemma \ref{for_Vakil}, we have
$$
\Delta^{N_0}=\Delta^{N_1}=\Delta^{N_2}=\Delta^{N_3}=1092,
\qquad
\Delta^{N_4}=1093.
$$
Similarly, Theorem \ref{23} c) can be seen through the pairs
$(a_0,k_0)=(47,96)$, $(a_1,k_1)=(48,97)$, $(a_2,k_2)=(48,98)$, $(a_3,k_3)=(49,99)$, $(a_4,k_4)=(48,100)$.
\end{ex}

\section{Proof of the main theorems}\label{proof}
\subsection{Proof of Theorem \ref{main-1}}
In this subsection, we use Lemma \ref{8} to prove Theorem \ref{main-1}.
\begin{proof} Let $\overline{n}=[\alpha_k,\cdots,\alpha_1]$ with $\alpha_k \geq 1$. 
\begin{enumerate}
\item For $k=1$ and $\alpha_1\geq 1$, we have $[\alpha_1]\rightarrow [\alpha_1-1]$ and $f([1])=1$. Hence, $f([\alpha_1])=\alpha_1.$

\item For $k=2$ and $\alpha_2,\alpha_1\geq 1$, we have
\begin{small}
\begin{align*}
[\alpha_2,\alpha_1]&\rightarrow[\alpha_2-1,\alpha_2+1]\text{ or } [\alpha_2,\alpha_1-1],\\
[\alpha_2,0_1]&\rightarrow [\alpha_2-1,1]\text{ or }[\alpha_2-1].
\end{align*}
\end{small}
So, we only consider two cases ($\alpha_1\geq 0$):
\begin{small}
\begin{align*}
[\alpha_2,\alpha_1]&\rightarrow \cdots \rightarrow [0_1,\alpha_1']=[\alpha_1'],\\
[\alpha_2,\alpha_1]&\rightarrow \cdots \rightarrow [\alpha_2',0_1]\rightarrow[\alpha_2'-1].
\end{align*}
\end{small}
Thus, $f([\alpha_2,\alpha_1])$ equals
\begin{small}
\begin{align*}
&\max \begin{cases}
2(\alpha_2-\alpha_2')+\alpha_1+1+f([\alpha_2'-1]),\\
2\alpha_2+\alpha_1-\alpha_1'+f([\alpha_1'])
\end{cases}\\
=\,&\max\begin{cases}
2\alpha_2+\alpha_1-\alpha_2',\\
2\alpha_2+\alpha_1
\end{cases}\\
=\,&2\alpha_2+\alpha_1.
\end{align*}
\end{small}

\item For $k=3$, similarly to the case $k=2$, we only consider four cases:
\begin{small}
\begin{align*}
[\alpha_3,\alpha_2,\alpha_1]\rightarrow \cdots &\rightarrow [0_1,\alpha_2',\alpha_1']=[\alpha_2',\alpha_1'],\\
[\alpha_3,\alpha_2,\alpha_1]\rightarrow \cdots &\rightarrow [\alpha_3',0_1,\alpha_1']\rightarrow[\alpha_3'-1,\alpha_1'+2],\\
[\alpha_3,\alpha_2,\alpha_1]\rightarrow \cdots &\rightarrow [\alpha_3',\alpha_2',0_1]\rightarrow[\alpha_3',\alpha_2'-1],\\
[\alpha_3,\alpha_2,\alpha_1]\rightarrow \cdots &\rightarrow [\alpha_3',0_1,0_2]\rightarrow[\alpha_3'-1].
\end{align*}
\end{small}
Thus, $f([\alpha_3,\alpha_2,\alpha_1])$ equals
\begin{small}
\begin{align*}
&\max\begin{cases}
3\alpha_3+2(\alpha_2-\alpha_2')+(\alpha_1-\alpha_1')+f([\alpha_2',\alpha_1']),\\
3(\alpha_3-\alpha_3')+2\alpha_2+(\alpha_1-\alpha_1')+1+f([\alpha_3'-1,\alpha_1'+2]),\\ 
3(\alpha_3-\alpha_3')+2(\alpha_2-\alpha_2')+\alpha_1+1
+f([\alpha_3',\alpha_2'-1]),\\
3(\alpha_3-\alpha_3')+2\alpha_2+\alpha_1+1+
f([\alpha_3'-1])
\end{cases}\\
=\,&\max\begin{cases}
3\alpha_3+2\alpha_2+\alpha_1,\\
3\alpha_3+2\alpha_2+\alpha_1+1-\alpha_3',\\
3\alpha_3+2\alpha_2-\alpha_2'-\alpha_3',\\
3\alpha_3-2\alpha_3'
\end{cases}\\
=\,&3\alpha_3+2\alpha_2+\alpha_1.
\end{align*}
\end{small}

\item For $k=4$, we only consider the following cases:
\begin{small}
\begin{align*}
[\alpha_4,\cdots,\alpha_1]\rightarrow \cdots&\rightarrow [0_1,\alpha_3',\alpha_2',\alpha_1'], \\
[\alpha_4,\cdots,\alpha_1]\rightarrow \cdots&\rightarrow [\alpha_4',0_1,\alpha_2',\alpha_1'] \rightarrow [\alpha_4'-1,\alpha_2'+2,\alpha_1'],\\
[\alpha_4,\cdots,\alpha_1]\rightarrow \cdots&\rightarrow [\alpha_4',\alpha_3',0_1,\alpha_1'] \rightarrow [\alpha_4',\alpha_3'-1,\alpha_1'+2],\\
[\alpha_4,\cdots,\alpha_1]\rightarrow \cdots&\rightarrow [\alpha_4',\alpha_3',\alpha_2',0_1] \rightarrow [\alpha_4',\alpha_3',\alpha_2'-1],\\
[\alpha_4,\cdots,\alpha_1]\rightarrow \cdots&\rightarrow [\alpha_4',0_1,0_2,\alpha_1'] \rightarrow [\alpha_4'-1,\alpha_1'+3],\\
[\alpha_4,\cdots,\alpha_1]\rightarrow \cdots&\rightarrow [\alpha_4',\alpha_3',0_1,0_2] \rightarrow [\alpha_4',\alpha_3'-1],\\
[\alpha_4,\cdots,\alpha_1]\rightarrow \cdots&\rightarrow [\alpha_4',0_1,0_2,0_3] \rightarrow [\alpha_4'-1].
\end{align*}
\end{small}
Thus, $f([\alpha_4,\cdots,\alpha_1])$ equals
\begin{small}
\begin{align*}
&\max\begin{cases}
4\alpha_4+\sum\limits_{j=1}^3 j(\alpha_j-\alpha_j')+f([\alpha_3',\alpha_2',\alpha_1']),\\
4(\alpha_4-\alpha_4')+3\alpha_3+2(\alpha_2-\alpha_2')+(\alpha_1-\alpha_1')+1+f([\alpha_4'-1,\alpha_2'+2,\alpha_1']),\\
4(\alpha_4-\alpha_4')+3(\alpha_3-\alpha_3')+2\alpha_2+(\alpha_1-\alpha_1')+1+f([\alpha_4',\alpha_3'-1,\alpha_1'+2]),\\
4(\alpha_4-\alpha_4')+3(\alpha_3-\alpha_3')+2(\alpha_2-\alpha_2')+\alpha_1+1+f([\alpha_4',\alpha_3',\alpha_2'-1]),\\
4(\alpha_4-\alpha_4')+3\alpha_3+2\alpha_2+(\alpha_1-\alpha_1')+1+f([\alpha_4'-1,\alpha_1'+3]),\\
4(\alpha_4-\alpha_4')+3(\alpha_3-\alpha_3')+2\alpha_2+\alpha_1+1+f([\alpha_4',\alpha_3'-1]),\\
4(\alpha_4-\alpha_4')+3\alpha_3+2\alpha_2+\alpha_1+1+f([\alpha_4'-1])
\end{cases}\\
=\,&\max\begin{cases}
\sum\limits_{j=1}^4 j\alpha_j,\\
\sum\limits_{j=1}^4 j\alpha_j-\alpha_4'+2,\\
\sum\limits_{j=1}^4 j\alpha_j-(\alpha_4'+\alpha_3')+1,\\
\sum\limits_{j=1}^4 j\alpha_j-(\alpha_4'+\alpha_3'+\alpha_2'),\\
\sum\limits_{j=1}^4 j\alpha_j-2\alpha_4'+2,\\
\sum\limits_{j=1}^4 j\alpha_j-2(\alpha_4'+\alpha_3'),\\
\sum\limits_{j=1}^4 j\alpha_j-3\alpha_4'
\end{cases}\\
=\,&\sum\limits_{j=1}^4 j\alpha_j+1.
\end{align*}
\end{small}

\item[5.] For $k=5,6,7$, the arguments are similar.
\end{enumerate}
\end{proof}

\subsection{Proof of Theorem \ref{main-2}}

\begin{lemm} \label{Delta_n_geq_log2i}
For $n\geq \log_2 i$, $i\geq 2$, we have
\begin{equation*}
    \Delta^{[n,0_1,\dots, 0_{i-1}]} = \Delta^{[\lfloor \log_2 i\rfloor -1, 0_1,\dots, 0_{i-1}]}.
\end{equation*}
\end{lemm}
\begin{proof}
Denote $[n,0_1,\dots,0_{i-1}]$ by $(n,i)'$. Since $2^i(2^n-1)$ is not a Vakil number, Theorem \ref{18} implies that in the canonical path of $T_{(n,i)'}$ we have
\[
(n,i)'\xrightarrow{\makebox[1.5cm]{i steps}}(n-1,i)'\xrightarrow{\makebox[1.5cm]{i steps}}\dots \xrightarrow{\makebox[1.5cm]{i steps}}(n',i)',
\]
where $1\leq n' \overset{+}{\approx} \log_2 i$. The number $2^i(2^{n'}-1)$ is the first Vakil number encountered in the reduction process from $(n,i)'$, with
\begin{equation*}
    n'=\begin{cases}
    \lfloor \log_2 i \rfloor & \text{if } \log_2 i \notin \mathbb{N},\\
    \lfloor \log_2 i \rfloor-1 & \text{if } \log_2 i \in \mathbb{N}.
    \end{cases}
\end{equation*}
By Remark \ref{f_by_cano}, we have
\[
\Delta^{[n,0_1,\dots,0_{i-1}]}=\Delta^{[n',0_1,\dots,0_{i-1}]}.
\]
In the case $\log_2 i \notin \mathbb{N}$, by Corollary \ref{f_n_leq_log2i}, we obtain
\[
\Delta^{[n'-1,0_1,\dots,0_{i-1}]}= 2^{n'+1}+\frac{i(i-5)}{2}-1-n'i
= \Delta^{[n',0_1,\dots,0_{i-1}]},
\]
since the number $k\geq 1$ such that $i\approx (2^{n'-1}-1)2^k+k+1$ is $k=2$, and the number $k\geq 1$ such that $i\approx (2^{n'}-1)2^k+k+1$ is $k=1$.
\end{proof}

\begin{ex}
In Example \ref{ex_reduction}, we see that $\Delta^{[2,0,0]} = \Delta^{[1,0,0]} = 0$.
\end{ex}

We now proceed to the proof of Theorem \ref{main-2}.
\begin{proof}
By Lemma \ref{exist_path}, we have
\begin{equation*}
[\sum\limits_{j=1}^k \alpha_j, 0_1,\dots,0_{k-1}] \leadsto [\alpha_k,\dots,\alpha_1] \leadsto [m,0_1,\dots,0_{k-1}].
\end{equation*}
By Lemmas \ref{inequality_Delta} and \ref{Delta_n_geq_log2i}, we have
\begin{equation*}
    \Delta^{[m,0_1,\dots,0_{k-1}]} = \Delta^{[\sum\limits_{j=1}^k \alpha_j, 0_1,\dots,0_{k-1}]} \geq \Delta^{[\alpha_k,\dots,\alpha_1]} \geq \Delta^{[m,0_1,\dots,0_{k-1}]}.
\end{equation*}
Hence, by Corollary \ref{f_n_leq_log2i}, we obtain
\begin{equation*}
    \Delta^{[\alpha_k,\dots,\alpha_1]} = \Delta^{[m,0_1,\dots,0_{k-1}]} = (2^m-1)2^h+\frac{(k-h)(k-h-1)}{2} -mk,
\end{equation*}
where $h\geq 1$ is such that $(2^m-1)2^{h-1}+h+1 \approx k$, or equivalently, $k \approx (2^m-1)2^h+h+1$.
\end{proof}

\subsection{Proof of Theorem \ref{main-0}}
\begin{proof}
By Remark \ref{f_by_cano} and Lemma \ref{8}, we have $\Delta^n=\Delta^{\overline{n}}=\Delta^{n'}$, where $n'$ is the closest Vakil number to $\overline{n}$. By Theorem \ref{23}, $\Delta^{n'} = \Delta^{\widetilde{n}}$, where $\widetilde{n}$ is the closest Vakil number to $n'$ with Vakil pair $(a,k)$, $4 \mid k$. Indeed, with the notation in Theorem \ref{23}, the binary class representations of $N_i$ are
\begin{align*}
N_0&=[a,\dots ,b, c, 0,\dots 0],\\
N_1&=[a,\dots ,b, c,1,\dots 0],\\
N_2&=[a,\dots ,b, c+1,\dots 0],\\
N_3&=[a,\dots ,b, c+2,\dots 0],\\
N_4&=[a,\dots ,b+1, 0, 0\dots 0],
\end{align*}
where $k_0 =[a,\dots ,b,c,0]$. It is clear that if $n'$ is one of $N_0, N_1, N_2, N_3$, then $\widetilde{n}=N_0$, and if $n'=N_4$, then $\widetilde{n}=N_4$. It is also clear that $\widetilde{n}$ is exactly the closest Vakil number to $\overline{n}$ with Vakil pair $(a,k)$, $4 \mid k$. Thus, we have proven the theorem.
\end{proof}

We can describe a way to compute $f(n)$ using Theorem \ref{main-0} as follows:
\begin{itemize}
\item[1.] Suppose that $\overline{n}$ has dimension $d$. Let $x=\lceil \log_2 d \rceil$.
\item[2.] First, we create a short table of Vakil pairs of dimension $d$:
\begin{itemize}
    \item If $2^{x-1}+x<d$, then we start with the first Vakil pair $V(2^d)=(d-x-1,2^x)$. By Lemma \ref{for_Vakil}, we can compute $f(2^d)$ and $\Delta^{2^d}$. By Theorem \ref{26}, to find the remaining Vakil pairs of dimension $d$ and the corresponding values of $\Delta$, we add $4$ to $2^x$ as long as the sum does not exceed $2^{x+1}$. We use Theorem \ref{23} to obtain the new pair $(a_4,k_4)$ from $(a_0,k_0)$. We compute the Vakil pair $(a',k')$ for each number $2^{a+1}k$ and skip all Vakil pairs $(a',k')$ such that $4 \nmid k'$. The corresponding value of $\Delta$ is given by the second and third parts of Theorem \ref{23}. Note that these are inferred directly from Lemma \ref{for_Vakil}, but Theorem \ref{23} provides a faster method.
    \item If $2^{x-1}+x\geq d$, we do the same thing with the Vakil pair $V(2^d)=(d-x,2^{x-1})$, while adding $4$ does not exceed $2^x$. In this case, we know that all pairs we obtain are Vakil pairs by the second part of Theorem \ref{26}.
\end{itemize}
\item[3.] Now, we use the table in Step 2 to compute $f(n)$. By Lemma \ref{exist_path}, we determine the list of classes that $\overline{n}$ can reach in the table. The closest class to $\overline{n}$ is $\widehat{n}$. Then $\Delta^{\overline{n}} = \Delta^{\widehat{n}}$, and $f(n)=S(\overline{n})+\widehat{n}$.
\end{itemize}

\section{Examples}\label{example}

We give an example of how to compute $f(n)$ using Theorem \ref{main-0}.
\begin{ex}
Let $\overline{n}=[1,1,2,1,3^{2^7},0,0,0,2^9,4,\dots,0]$ be of dimension $53$. We compute $f(n)$ as follows.
\begin{itemize}
    \item[1.] The dimension of $\overline{n}$ is $53$. We have $\lceil \log_2 53 \rceil = 6$.
    \item[2.] We create a short table of Vakil pairs of dimension $53$:
    \begin{itemize}
        \item We have $2^5+6=38<53$. Hence, $V(2^{53})=(46,64)$. By Lemma \ref{for_Vakil}, we have $f(2^{53}) = 64 + \frac{46\cdot 47}{2}$ and
        \[
        \Delta^{2^{53}}=f(2^{53})-S(2^{53})=1092.
        \]
        \item By Theorem \ref{26}, to find the remaining Vakil pairs of dimension $53$ and the corresponding values of $\Delta$, we add $4$ to $k$ until we reach $2^7-4$. The corresponding value of $a$ for $k+4$ is then computed by the first part of Theorem \ref{23}. We compute the Vakil pair $(a',k')$ for each number $2^{a+1}k$ and skip all Vakil pairs $(a',k')$ such that $4\nmid k'$. The corresponding value of $\Delta$ is given by Theorem \ref{23}.
        \begin{small}
        \begin{table}[ht!]
\centering{
\begin{tabular}{|c|l|l|l|l|l|}
\hline
$a$&$k$&$t_{k/4}$&Vakil pair&$\Delta$&representation of $2^{a+1}k$\\
\hline
46&64&0&(46,64) &$64+\frac{46\cdot47}{2}-53=1092$&[1,0,\dots,0]\\
47&68&1&(47,68)&$1092+\frac{1\cdot2}{2}=1093$&[1,0,0,1,0,\dots,0]\\
47&72&0&(47,72)&$1093+\frac{2\cdot3}{2}=1096$&[1,0,1,0,\dots,0]\\
48&76&2&(48,76)&$1096+\frac{1\cdot2}{2}=1097$&[1,0,2,0,\dots,0]\\
47&80&0&(47,80)&$1097+\frac{3\cdot4}{2}=1103$&[1,1,0,\dots,0]\\
48&84&1&(48,84)&$1103+\frac{1\cdot2}{2}=1104$&[1,1,1,0,\dots,0]\\
48&88&0&(48,88)&$1104+\frac{2\cdot3}{2}=1107$&[1,2,0,\dots,0]\\
49&92&3&(49,92)&$1107+\frac{1\cdot2}{2}=1108$&[1,3,0,\dots,0]\\
47&96&0&(48,48)&$1108+\frac{4\cdot5}{2}+48-48=1118$&[2,0,\dots,0]\\
48&100&1&(49,50), skip& &\\
48&104&0&(49,52)&$1118+\frac{1\cdot2}{2}=1119$&[2,1,\dots,0]\\
49&108&2&(50,54), skip& &\\
48&112&0&(49,56)&$1119+\frac{1\cdot2}{2}+49-56=1122$&[3,0,\dots,0]\\
49&116&1&(50,58), skip& &\\
49&120&0&(50,60)&$1122+\frac{1\cdot2}{2}=1123$&[4,0,\dots,0]\\
50&124&skip&(51,62), skip& &\\
\hline
\end{tabular}}
\caption{\label{tab_53} Short table of $\Delta$ for $53$-dimensional binary classes.}
\end{table}
\end{small}
    \end{itemize}
    \item[3.] We use Table \ref{tab_53} to compute $f(n)$. By Lemma \ref{exist_path} and the last column of Table \ref{tab_53}, we know that $\overline{n}$ can reach only the first six classes in the table. The closest class to $\overline{n}$ is $[1,1,1,0,\dots,0]$. Hence, $\Delta^{\overline{n}} = 1104$, and $f(n)=S(\overline{n})+1104$.
\end{itemize}
\end{ex}

The example below shows how to use Lemma \ref{for_Vakil} and Theorems \ref{main-1} and \ref{main-2}. Although Theorem \ref{main-0} is powerful, it should be considered only when we cannot apply the previous results.

\begin{ex} \hfill
\begin{enumerate}
\item[1.] If $n=69632$, then $n=\overline{n}=2^{12}\cdot 17$ is a Vakil number with Vakil pair $(11,17)$. Hence, by Lemma \ref{for_Vakil}, we have
\[
f(69632)=17+\frac{11\cdot 12}{2}=83.
\]

\item[2.] If $n=473$, then $n=111011001_{(2)}$ and $\overline{n}=11101100_{(2)}=236=2^2\cdot 59$. We see that $473$ is not a Vakil number because $59>3$. Since $236=[3,2,0]$ has dimension $3$, by Theorem \ref{main-1}, we have
\[
f(473)=S([3,2,0])=13.
\]

\item[3.] If $n=8923773549686799$, then
\begin{align*}
n&=11111101101000001111111110000011111000011000000001111_{(2)},\\
\overline{n}&=1111110110100000111111111000001111100001100000000_{(2)}\\
&=557735846855424=2^8\cdot 2178655651779.
\end{align*}
Hence, $n$ is not a Vakil number. We see that
\[
\overline{n}=[6,2,1,0,0,0,0,9,0,0,0,0,5,0,0,0,2,0,0,0,0,0,0,0]
\]
has dimension $24$. Since $\alpha_{24}=6\geq \lfloor \log_2 24 \rfloor-1=3$, by Theorem \ref{main-2}, we have
\[
f(n)=S(\overline{n})+(2^3-1)2^2+\frac{(24-2)(24-3)}{2}-3\cdot 24=628.
\]

\item[4.] If $n=12737511856113$, then
\begin{align*}
n&=10111001010110101110110111101001011111110001_{(2)},\\
\overline{n}&=1011100101011010111011011110100101111111000_{(2)}\\
&=6368755928056=2^3\cdot 796094491007.
\end{align*}
Hence, $n$ is not a Vakil number. We see that
\[
\overline{n}=[1,3,0,1,1,2,1,3,2,4,1,0,1,7,0,0]
\]
has dimension $16$, and $\alpha_{16}=1<\lfloor \log_2 16 \rfloor-1$. Therefore, we cannot apply Lemma \ref{for_Vakil} or Theorems \ref{main-1} and \ref{main-2} directly, and we use Theorem \ref{main-0}. Since $\lceil \log_2 16\rceil=4$ and $2^3+4=12<16$, we have $V(2^{16})=(11,16)$.

\begin{table}[ht!]
\centering{
\begin{tabular}{|c|l|l|l|l|l|}
\hline
$a$&$k$&$t_{k/4}$&Vakil pair&$\Delta$&representation of $2^{a+1}k$\\
\hline
$11$&$16$&$0$&$(11,16)$&$16+\frac{11\cdot 12}{2}-16=66$&$[1,0,\dots,0]$\\
$12$&$20$&$1$&$(12,20)$&$66+\frac{1\cdot 2}{2}=67$&$[1,1,0,\dots,0]$\\
$12$&$24$&$0$&$(12,24)$&$67+\frac{2\cdot 3}{2}=70$&$[2,0,\dots,0]$\\
$13$&$28$&skip&$(14,14)$ (skip)&&$[3,0,\dots,0]$\\
\hline
\end{tabular}}
\caption{\label{tab_16} Short list table of $\Delta$ for $16$-dimensional binary classes.}
\end{table}

By Lemma \ref{exist_path}, $\overline{n}$ can reach only the first two classes in Table \ref{tab_16}. The closest class to $\overline{n}$ is $[1,1,0,\dots,0]$. Hence, $\Delta^{\overline{n}}=67$ and
\[
f(n)=S(\overline{n})+67=287.
\]
\end{enumerate}
\end{ex}

\bibliography{references}{}

\begin{thebibliography}{{The}24}

\bibitem[Chr98]{christensen1998ideals}
J~Daniel Christensen.
\newblock Ideals in triangulated categories: phantoms, ghosts and skeleta.
\newblock {\em Advances in Mathematics}, 136(2):284--339, 1998.

\bibitem[{The}24]{SageMath}
{The Sage Developers}.
\newblock {\em SageMath, the Sage Mathematics Software System (Version 10.x)}.
\newblock SageMath, 2024.

\bibitem[Vak99]{vakil1999steenrod}
Ravi Vakil.
\newblock On the {S}teenrod length of real projective spaces: finding longest chains in certain directed graphs.
\newblock {\em Discrete mathematics}, 204(1-3):415--425, 1999.

\end{thebibliography}
\bibliographystyle{alpha}

\noindent Institut f\"ur Algebra und Geometrie, Otto-von-Guericke-Universit\"at Magdeburg, Germany\\
E-mail: \href{khanh.mathematic@gmail.com}{khanh.mathematic@gmail.com} \\

\end{document}